%% file: main.tex
\documentclass[12pt]{article}
\newcommand*{\inserttitle}{Quantitative selection theorems}
\input{header}

\usepackage{apptools}
\AtAppendix{\counterwithin{theorem}{section}}

\newcommand{\vol}{\operatorname{vol}}

\newcommand{\set}[1]{\underline{\mathcal{#1}}} % for a set whose elements are sets of sets

\usepackage{setspace}
\begin{document}

\thispagestyle{plain}

% {\centering
%     {\Large \scshape \inserttitle}\\[1em]
%     {\large Travis Dillon}\\
% }

\customtitle{\inserttitle}
{\centering
    {\large\scshape Travis Dillon\par}
    {\small Massachusetts Institute of Technology}\par
    % \vspace{0.1in}
}

\begin{abstract}\setstretch{1.0}
    The point selection theorem says that the convex hull of any finite point set contains a point that lies in a positive proportion of the simplices determined by that set. This paper proves several new volumetric versions of this theorem which replace the points by sets of large volume, including the first volumetric selection theorem for $(d+1)$-tuples. As consequences, we significantly decrease the upper bound for the number of sets necessary in a volumetric weak $\epsilon$-net, from $O_d(\epsilon^{-d^2(d+3)^2/4})$ to $O_d(\epsilon^{-(d+1)})$, and substantially reduce the the piercing number for volumetric $(p,q)$-theorems.

    We also prove a volumetric version of the homogeneous point selection theorem. To do so, we introduce a volumetric same-type lemma and a new volumetric colorful Tverberg theorem. 
    
    We prove all of our results for diameter as well as volume.
\end{abstract}

\section{Overview}

Any point in the convex hull of a point set is contained in a simplex determined by that set; that is the statement of Carath\'eodory's theorem from 1907 \cite{caratheodory}. However, according to the \emph{point selection theorem}, some points in the convex hull are contained in many more simplices:

\begin{theorem*}[B\'ar\'any \cite{selection-thm-original}, Theorem 5.1]
    For any finite point set $X \subseteq \R^d$, there is a point $z \in \R^d$ that is contained in at least $c_d \binom{|X|}{d+1}$ simplices whose vertices lie in $X$, where the constant $c_d > 0$ depends only on $d$.
\end{theorem*}

Are there many ``deep'' points? If ``many'' is measured by volume, then not necessarily: The volume of $\conv(X)$ itself may be arbitrarily small. Even if $\vol\!\big(\!\conv(X)\big) = 1$, it may be that all but $1$ point lie in an arbitrarily small ball, which any deep point must also lie within.

However, these questions have better answers if we replace the points in $X$ by sets with fixed volume:

\begin{theorem*}[Selection theorem with volume; Jung, Nasz\'odi \cite{quant-frac-pq}, Lemma 5.3]
    For any finite family $\mathcal F$ of volume-1 convex sets in $\R^d$, there is a set $E$ with volume at least $d^{-d}$ that is contained in $\conv(\mathcal G)$ for at least $c_d \binom{|\mathcal F|}{d^2(d+3)^2/4}$ choices of $\mathcal G \in \binom{\mathcal F}{d^2(d+3)^2/4}$.\footnote{Each appearance of $c_d$ in a theorem statement is independent and indicates a positive constant depending only on $d$. The constant $c_d$ in this theorem and the $c_d$ in the point selection theorem represent different real numbers.\label{foot:c_d}}
\end{theorem*}

However, this selection theorem is $\frac{1}{4} d^2(d+3)^2/4$-tuples, not $(d+1)$-tuples. Rolnick and Sober\'on \cite{quant-p-q-theorems} have also proven a version of this theorem, with a greater volume for $E$ but even larger tuples.

The focus of this article is to improve various aspects of the volumetric selection theorem. Our first main result is a selection theorem for $(d+1)$-tuples:

\begin{theorem}\label{thm:selection-d+1}
    For any finite family $\mathcal F$ of volume-1 convex sets in $\R^d$, there is a set with volume at least $4^{-d^2 (1 + o_d(1))}$ that is contained in $\conv(\mathcal G)$ for at least $c_d \binom{|\mathcal F|}{d+1}$ choices of $\mathcal G \in \binom{\mathcal F}{d+1}$.
\end{theorem}

The point of \cref{thm:selection-d+1} is to optimize the subset size: $d+1$ cannot be reduced. As a result, the volume of the intersection is extremely small. As Rolnick and Sober\'on \cite{quant-p-q-theorems} showed, as the subset size is increased, so is the volume. We add two more selection results which also exhibit this behavior. First, with only a few more vertices, the volume of the intersection  can be drastically increased:

\begin{theorem}\label{thm:selection-2d}
    For any finite family $\mathcal F$ of convex sets in $\R^d$, each with volume at least 1, there is a set of volume at least $(5d^3)^{-d}$ that is contained in $\conv(\mathcal G)$ for at least $c_d \binom{|\mathcal F|}{2d}$ choices of $\mathcal G \in \binom{\mathcal F}{2d}$.
\end{theorem}

Moreover, the volume guarantee of Jung and Nasz\'odi's result can be exactly matched with tuples of quadratic size:

\begin{theorem}\label{thm:selection-d(d+3)/2}
    For any finite family $\mathcal F$ of convex sets in $\R^d$, each with volume at least 1, there is a set $E$ of volume at least $d^{-d}$ that is contained in $\conv(\mathcal G)$ for at least $c_d \binom{|\mathcal F|}{\frac{1}{2} d(d+3) +1}$ choices of $\mathcal G \in \binom{\mathcal F}{\frac{1}{2} d(d+3) +1}$.
\end{theorem}

The second main result in this paper is a volumetric version of the \emph{homogeneous selection theorem}, a remarkable strengthening of the selection lemma.

\begin{theorem*}[Homogeneous selection theorem; Pach \cite{homogeneous-selection}]
    For any sufficiently large point set $X$ in $\R^d$, there are disjoint subsets $Y_1,\dots, Y_{d+1} \subseteq X$ of size $|Y_i| \geq c_d n$ and a point $z \in \R^d$ such that $z \in \conv(y_1,\dots,y_{d+1})$ for every selection $y_i \in Y_i$.
\end{theorem*}

In other words, $z$ is not only contained in a constant proportion of the simplices, but all the simplices determined by $d+1$ large subsets! We extend this selection theorem to a volumetric version.

\begin{theorem}[Homogeneous selection with volume]\label{thm:vol-hom-sel}
    For any large enough family $\mathcal F$ of volume-1 convex sets in $\R^d$, there are $d+1$ disjoint subfamilies $\mathcal G_1, \dots, \mathcal G_{d+1}$ of $\mathcal F$, each containing at least $c_d |\mathcal F|$ sets, and a convex set $E \subseteq \R^d$ of volume $v_d > 0$, such that $E \subseteq \conv(G_1,\dots,G_{d+1})$ for every selection $G_i \in \mathcal G_i$. While $v_d$ depends on the dimension $d$, it does not depend on $\mathcal F$.
\end{theorem}

Actually, we prove a colorful version of the homogeneous selection theorem (\cref{thm:colorful-hom-vol-sel}), from which \cref{thm:vol-hom-sel} follows. The crux of our proof is a new colorful Tverberg theorem for volume (\cref{thm:col-tv-vol-smaller}). In \cref{sec:quant-extensions}, where we extend our results to other quantitative parameters, we prove the first colorful Tverberg theorem for diameter.

\cref{thm:vol-hom-sel} is optimal in that $d+1$ cannot be replaced by any smaller integer.
%(For the point selection lemma, any set $X$ that is general position is a counterexample. To get a counterexample for the volumetric version, scale $X$ so that the points are far apart from each other and place a ball of volume 1 centered at each point in $X$.)

\subsection*{Applications}

An especially important application of the point selection lemma is the \emph{weak $\epsilon$-net theorem}, which has applications throughout discrete and computational geometry. (See \cite{matousek-eps-net-survey} for a survey.) In short, it says that the set of convex hulls of large subsets of a point set $X$ in $\R^d$ have a transversal whose size is independent of $|X|$. More precisely: 

\begin{theorem*}[Weak $\epsilon$-nets; Alon--B\'ar\'any--F\"uredi--Kleitnman \cite{epsilon-nets-original}]
    For every $\epsilon > 0$, there is a constant $c(\epsilon, d)$ such that for any finite point set $X \subseteq \R^d$, there is a set $P \subseteq \R^d$ of at most $c(d,\epsilon)$ points such that $\conv(Y)$ contains a point of $P$ whenever $Y$ contains at least $\epsilon |X|$ sets in $X$.
\end{theorem*}

Recently, a few papers \cite{quant-p-q-theorems, quant-frac-pq, improved-quant-frac-helly} have introduced and applied \emph{quantitative} versions of this theorem, which produce a transversal by sets of positive volume.

\begin{theorem*}[Jung--Nasz\'odi \cite{quant-frac-pq}]\label{jung-eps-net}
    For every finite family $\mathcal F$ of volume-1 convex sets, there is a family $\mathcal S$ of at most $O_d(\epsilon^{-d^2(d+3)^2/4})$ sets, each of volume at least $d^{-d}$, such that $\conv(\mathcal G)$ contains a set in $\mathcal S$ whenever $\mathcal G$ contains at least $\epsilon |\mathcal F|$ sets in $\mathcal F$.
\end{theorem*}

A main research topic related to weak $\epsilon$-nets for points is to establish better bounds on $c(d,\epsilon)$. The original bound \cite{epsilon-nets-original} is slightly better than $O_d(\epsilon^{-(d+1)})$; after a series of works \cite{eps-net-upper-d, matousek-eps-net-upper}, the current best general upper bound is $c(d,\epsilon) = o_{\epsilon}(\epsilon^{-(d-1/2)})$ when $d \geq 4$ \cite{Rubin-epsilon-net-upper-bound}. (Somewhat better bounds are known for particular small dimensions.) The current best lower bound is $\Omega_d(\epsilon^{-1} \log(1/\epsilon)^{d-1})$ \cite{eps-net-lower-bound}.

Each of the selection theorems in the previous section implies a weak $\epsilon$-net theorem for volume, with varying trade-offs between the number of piercing sets and the volume of those sets. The three results can be summarized as:

\begin{theorem}[Weak $\epsilon$-nets for volume]\label{thm:vol-eps-nets}
    Let $\epsilon > 0$. For each of the three choices
    \[\begin{array}{r @{{}={}\big(\,} c @{,\ \,} c @{\big)}}
        (v,\alpha) & d^{-d} & \frac{1}{2}d(d+3)+1\\
        (v,\alpha) & (5d^3)^{-d} & 2d\\
        (v,\alpha) &  4^{-d^2(1+o(1))} & d+1
    \end{array}\]
    the following statement is true. For every finite family $\mathcal F$ of volume-1 convex sets, there is a family $\mathcal S$ of $O_d(\epsilon^{-\alpha})$ sets, each of volume at least $v$, such that $\conv(\mathcal G)$ contains a set in $\mathcal S$ whenever $\mathcal G$ contains at least $\epsilon |\mathcal F|$ sets in $\mathcal F$.
\end{theorem}

The exponent for the case $v=d^{-d}$ can be reduced to $\alpha = \frac{1}{2}d(d+3)-\frac{1}{2}$ by a different argument; we sketch this argument after proving \cref{thm:vol-eps-nets}.

The interest in a volumetric weak $\epsilon$-net theorem was first inspired by volumetric versions of the \emph{$(p,q)$-theorem}, whose proofs use volumetric weak $\epsilon$-nets as a crucial ingredient. A family $\mathcal F$ is \emph{$(p,q)$-intersecting} if every collection of $p$ sets in $\mathcal F$ contains $q$ sets whose intersection is nonempty. The classical $(p,q)$-theorem \cite{original-pq} states that \textit{for any $p \geq q \geq d+1$, there is a constant $c(p,q,d)$ such that for any $(p,q)$-intersecting family $\mathcal F$ of convex sets, there is a set $X$ of at most $c(p,q,d)$ points such that every set in $\mathcal F$ contains at least one point of $X$.} In other words, any $(p,q)$-intersecting family of convex sets has a piercing set whose size is bounded by $p$, $q$, and $d$, but is independent of the family itself. (For example, the equation $c(d+1,d+1,d) = 1$ is another way of stating Helly's theorem on the intersection of convex sets \cite{Helly}.)

A volumetric variant assumes that each set in $\mathcal F$ has volume at least 1 and replaces $P$ by a family $\mathcal S$ of sets, each of volume at least $v(d) > 0$. Rolnick and Sober\'on \cite{quant-p-q-theorems} proved a version with $v(d) = 1 - \delta$ but $q \gg d+1$; more recent work proved volumetric $(p,q)$ theorems for $q \geq 3d+1$ \cite{quant-frac-pq} and $q \geq d+1$ \cite{improved-quant-frac-helly}, though these results provide a much smaller $v(d)$. Improving the volumetric weak $\epsilon$-net theorem yields a direct improvement on the size of the transversal guaranteed by the volumetric $(p,q)$-theorem. For example, the $(p,q)$-theorem in \cite{quant-frac-pq}, whose results are easiest to quantify, yields a transversal $\mathcal S$ with at most $(cp)^{3d^6/4 + O(d^5)}$ sets (for some constant $c>0$). Replacing their weak $\epsilon$-net theorem with the net of $O_d(\epsilon^{-(d+1)})$ sets in \cref{thm:vol-eps-nets} improves the transversal size to $(cp)^{3d^3 + O(d^2)}$.

\subsection*{Related work}

This area of volumetric combinatorial geometry can be traced back to a 1982 paper of B\'ar\'any, Katchalski, and Pach \cite{Barany:1982ga} which proved quantitative versions of Helly's, specifically that \textit{if the intersection of every $2d$ or fewer elements of a finite family $\mathcal F$ of convex sets in $\R^d$ has volume greater than or equal to 1, then the volume of $\bigcap \mathcal F$ is at least $d^{-2d^2}$.} They conjectured that $\vol\big(\bigcap \mathcal F\big) \geq (Cd)^{-c d}$ for some constants $c, C > 0$. Nasz\'odi proved this conjecture 2016 with $c=2$\cite{Naszodi:2016he} (and Brazitikos improved this to $c=3/2$ \cite{brazitikos-vol-helly}). This breakthrough has renewed interest over the last decade in volumetric and quantitative theorems in combinatorial geometry.

Helly's theorem itself has a much broader reach, and it forms the basis for an astounding variety of research in geometry, topology, and combinatorics; for an overview of these many research directions, see the surveys \cite{Danzer-survey,barany-helly-survey,Soberon-Helly-survey}.

\subsection*{Structure of the paper}

In \cref{sec:selection}, we prove \cref{thm:selection-d(d+3)/2,thm:selection-2d,thm:selection-d+1,thm:vol-eps-nets}. \cref{thm:vol-hom-sel} is proven in \cref{sec:hom-sel}. 

The proofs in this paper are flexible enough to produce quantitative selection (and thus weak $\epsilon$-net) theorems for other parameters, such as diameter. \cref{sec:quant-extensions} discusses these variations, including a proof of a colorful Tverberg theorem for diameter, a result which has not yet appeared in the literature.

Throughout the paper, parameters with subscripts are always positive and their values depend only on the variables in the subscript. For example, ``volume at least $v_d$'' means the volume is at least a positive quantity that depends only on the dimension $d$. These parameters represent the same number within a single proof or discussion, but may stand for a different value in a different proof. (See footnote \ref{foot:c_d}.)

Finally, the proofs in this paper rely on many works in quantitative combinatorial geometry, as well as some structural hypergraph theory. To smooth the narrative, the statements of some of these results are housed in an appendix. Proofs citing these results may only state the immediate application, referring to the appendix for the complete statement and a citation to the original work.

\section{Selection theorems for weak \texorpdfstring{$\epsilon$}{epsilon}-nets}\label{sec:selection}

\subsection*{Selection theorems}

We  will start by proving \cref{thm:selection-d+1,thm:selection-2d,thm:selection-d(d+3)/2}, in reverse order. \cref{thm:selection-d(d+3)/2} is simplest and introduces ideas that we'll use in the proofs of \cref{thm:selection-d+1,thm:selection-2d}.

Our first tool is John's theorem, a key result in convex geometry that can transform general problems about volume of convex sets into problems about ellipsoids. We will use this theorem in each of the next three proofs.

\begin{theorem}[John \cite{john}]\label{thm:john}
    The volume of the largest ellipsoid that is contained in a convex set $C \subseteq \R^d$ is at least $d^{-d} \vol(C)$.
\end{theorem}

\cref{thm:selection-d(d+3)/2}, as it turns out, can be proven by a relatively short parametrization argument. Every ellipsoid $E \subseteq \R^d$ can be written as $A(B^d) + x$ for a unique positive-definite matrix $A \in \R^{d\times d}$ and vector $x \in \R^d$, where $B^d$ is the $d$-dimensional unit ball. So each ellipsoid in $\R^d$ corresponds to a point $(A,x) \in P\!D^d \times \R^d$, where $P\!D^d$ is the set of $d\times d$ positive-definite matrices. Since $P\!D^d$ is a convex subset of $\R^{d(d+1)/2}$, this provides a parametrization of the set of ellipsoids in $\R^d$ as a convex subset of $\R^{d(d+3)/2}$. To prove \cref{thm:selection-d(d+3)/2}, we apply the point selection theorem to this parametrized space.

\begin{proof}[Proof of \cref{thm:selection-d(d+3)/2}]
    By John's theorem, each set $F \in \mathcal F$ contains an ellipsoid $A_F (B^d) + x_F$ of volume at least $d^{-d}$; set $y_F \defeq (A_F, x_F) \in P\!D^d \times \R^d \subset \R^{d(d+3)/2}$ and $Y = \{y_F\}_{F \in \mathcal F}$. By the point selection theorem in $\R^{d(d+3)/2}$, there is a point $(A,x)$ that is in at least $c_d \binom{|\mathcal F|}{\frac{1}{2} d(d+3) + 1}$ simplices determined by $Y$. A short calculation confirms that if $(A,x) \in \conv\big\{(A_i,x_i)\}_{i=1}^m$, then $A(B^d) + x \subseteq \conv\big( A_i(B^d) + x_i \big)$. So $E\defeq A(B^d) + x \subseteq \R^d$ is an ellipsoid contained in $\conv(\mathcal G)$ for at least $c_d \binom{|\mathcal F|}{\frac{1}{2}d(d+3)+1}$ choices of $\mathcal G \in \binom{\mathcal F}{\frac{1}{2}d(d+3)+1}$. Moreover, since the determinant is log-concave on the set of positive-definite matrices, $\vol(E) \geq d^{-d}$.
\end{proof}

The proofs of the remaining selection theorems rely on the volumetric version of the \emph{fractional Helly theorem}, which says that whenever a positive proportion of $k$-tuples of sets in $\mathcal F$ are intersecting, there is an intersecting subfamily of $\mathcal F$ with a positive proportion of the sets. There are several versions that incorporate volume; the specific one we will use can be proven by another parametrization argument.\footnote{Strictly speaking, the parametrization argument gives the same result with $\frac{1}{2} d(d+3)+1$ instead of $\frac{1}{2} d(d+3)$. Sarkar, Xue, and Sober\'on use topological tools to reduce the subset size by 1.}

\begin{proposition}[Fractional Helly for ellipsoids; Sarkar--Xue--Sober\'on \cite{pablo-concave-funcs}, Theorem 5.2.1]\label{thm:vol-frac-helly}
    For every $\alpha > 0$, there is a $\beta_{\alpha,d} > 0$ such that the following holds for any finite family $\mathcal F$ of convex sets in $\R^d$. If $\bigcap \mathcal G$ contains an ellipsoid of volume 1 for at least $\alpha \binom{|\mathcal F|}{d(d+3)/2}$ choices of $\mathcal G \in \binom{\mathcal F}{d(d+3)/2}$, then there is a subfamily $\mathcal F' \subseteq \mathcal F$ with $|\mathcal F'|\geq \beta_{\alpha,d} |\mathcal F|$ such that $\bigcap \mathcal F'$ contains an ellipsoid of volume 1.
\end{proposition}

Jung and Nasz\'odi later provided a different proof of the same result \cite[Proposition 1.5]{quant-frac-pq}.

\begin{proof}[Proof of \cref{thm:selection-2d}]
    We will apply the fractional Helly theorem to the family $\mathcal C \defeq \big\{\! \conv(\mathcal G) : \mathcal G \in \binom{\mathcal F}{2d} \big\}$, providing an ellipsoid $\widetilde E$ and a positive proportion of $\mathcal G \in \binom{\mathcal F}{2d}$ for which $\widetilde E \subseteq \conv(\mathcal G)$. Thus, we need to show that a positive proportion of the $\frac{1}{2}d(d+3)$-tuples of sets in $\mathcal C$ intersect with large volume.
    
    From John's theorem, we know that each set in $\mathcal F$ contains an ellipsoid of volume at least $d^{-d}$. Another parametrization argument yields a Tverberg-type theorem for ellipsoids (\cref{thm:ellipse-tverberg}), which tells us that any $f(d)=\big(\frac{d(d+3)}{2}+1\big)(\frac{d(d+3)}{2}-1)+1$ of the ellipsoids can be partitioned into $r \defeq \frac{d(d+3)}{2}$ collections $\mathcal G_1,\dots, \mathcal G_{r}$ such that $\bigcap_{i=1}^{r} \conv(\mathcal G_i)$ contains an ellipsoid $E$ of volume $d^{-d}$. Our next step, the key new ingredient in this proof, is to reduce the size of each $\mathcal G_i$ so that it contains exactly $2d$ sets (and therefore lies in $\mathcal C$), while maintaining a large intersection.
    
    Take an affine transformation $A$ that maps $E$ to the unit ball $B^d$. For each $i \in [r]$, we have $B^d \subseteq \conv(A(F) : F \in \mathcal G_i)$. By a quantitative version of the Steinitz theorem (\cref{thm:quant-steinitz}), for each $i \in [r]$ there is a set of at most $2d$ points $X_i \subseteq A(\mathcal G_i)$ such that $\conv(X_i) \supseteq \frac{1}{5d^2} B^d$. Let $\mathcal G_i'$ denote a collection of $2d$ sets in $\mathcal G_i$ that cover $A^{-1}(X_i)$. Then $\conv(\mathcal G_i') \supseteq \conv(A^{-1}X_i) \supseteq \frac{1}{5d^2} E$ for each $i$.

    This shows that among any $f(d)$ sets in $\mathcal F$, there are $r$ disjoint subfamilies of size $2d$ such that the intersection of their convex hulls contains an ellipsoid of volume at least $(5d^3)^{-d}$. Each collection of $r$ families can be formed from at most $\binom{|\mathcal F|-2dr}{f(d)-2dr}$ different $f(d)$-tuples of sets from $\mathcal F$. Thus, the number of elements of $\binom{\mathcal C}{r}$ whose intersection contains an ellipsoid of volume $(5d^3)^{-d}$ is at least 
    \[
        \frac{\binom{|\mathcal F|}{f(d)}}{\binom{|\mathcal F|-2dr}{f(d)-2dr}}
        \geq \alpha_d \binom{\binom{|\mathcal F|}{2d}}{r}
        = \alpha_d \binom{|\mathcal C|}{r},
    \]
    for some $\alpha_d > 0$. By \cref{thm:vol-frac-helly}, there is a subset $\mathcal C' \subseteq \mathcal C$ with at least $\beta_d\, |\mathcal C| = \beta_d\, \binom{|\mathcal F|}{2d}$ sets such that $\bigcap \mathcal C'$ contains an ellipsoid of volume $(5d^{-3})^{-d}$. This is what we wanted to prove.
\end{proof}

To prove \cref{thm:selection-d+1}, we need to reduce the collection $\mathcal G_i$ even further, so that it contains only $d+1$ sets. To do so, we can no longer deal only with ellipsoids; we must parley with less docile convex sets. That is the following lemma's task.

\begin{lemma}\label{thm:vol-planes}
    For any finite sets $X_1,\dots,X_m \subseteq \R^d$, each with $k \geq d+1$ points, such that $\vol\big(\bigcap_{i=1}^m \conv(X_i)\big) \geq 1$, there are subsets $Y_i \subseteq X_i$ of size $|Y_i| = d+1$ such that $\vol\big(\bigcap_{i=1}^m \conv(Y_i)\big) \geq \big(m \binom{k}{d}\big)^{-d}$.
\end{lemma}
\begin{proof}
    Let $\mathcal H_i$ denote the set of all hyperplanes determined by $d$ points in $X_i$, and define $\mathcal H = \bigcup_{i=1}^m \mathcal H_i$. Then $|\mathcal H| \leq m \binom{k}{d}$, and these hyperplanes divide $\R^d$ into at most $\big(m \binom{k}{d}\big)^d$ connected regions (see, for example,  \cite[Proposition 6.1.1]{matousek-lectures}). Thus, there is a convex set $C \subseteq \bigcap_{i=1}^m \conv(X_i)$ with volume at least $\big(m \binom{k}{d}\big)^{-d}$ whose interior does not intersect any of the hyperplanes in $\mathcal H$. By Carath\'eodory's theorem, the simplices determined by the points of $X_j$ cover the set $\bigcap_{i=1}^m \conv(X_i)$. Because the interior of $C$ intersects no hyperplane determined by $X_j$, there is a simplex $Y_j \subseteq X_j$ such that $C \subseteq \conv(Y_j)$. Therefore $\vol\big(\bigcap_{i=1}^m \conv(Y_i) \big) \geq \vol(C) \geq \big(m \binom{k}{d}\big)^{-d}$.
\end{proof}

\begin{proof}[Proof of \cref{thm:selection-d+1}]
    In the proof of \cref{thm:selection-2d}, use \cref{thm:vol-planes} to find sets $Y_i \subseteq X_i$ of size $d+1$; then (re)define $\mathcal G_i'$ as a collection of $d+1$ sets in $\mathcal G_i$ that covers $A^{-1}(Y_i)$. There is a set $C$ with
    \[
        \vol(C)
        \geq \bigg( \frac{1}{2}d(d+3) \binom{2d}{d}\bigg)^{-d} \vol\Big(\frac{1}{5d^2} E\Big)
        \geq \Big(\frac{1}{5} d^5 2^{2d} \Big)^{-d}
    \]
    such that $\conv(\mathcal G_i') \supseteq C$ for every $i$. After taking the John ellipsoid inside $C$, the rest of the proof may be copied exactly, with every instance of $2d$ replaced by $d+1$.
\end{proof}

Since we will make use of this technique again in \cref{sec:hom-sel_quant-lemmas}, it will be helpful to record it as a lemma:

\begin{lemma}\label{thm:refining-vol-tv}
    For any families $\mathcal G_1,\dots, \mathcal G_k$ of sets in $\R^d$ such that $\bigcap_{i=1}^k \conv(\mathcal G_i)$ has volume at least 1, there are subfamilies $\mathcal G_i' \subseteq \mathcal G_i$, each with at most $d+1$ sets, such that $\bigcap_{i=1}^k \conv(\mathcal G_i')$ has volume at least $(\frac{1}{5} d^3 k 2^{2d} )^{-d}$
\end{lemma}

\subsection*{Weak \texorpdfstring{$\epsilon$}{epsilon}-nets}

Our proof of \cref{thm:vol-eps-nets} follows the original algorithmic proof for existence of weak $\epsilon$-nets in \cite{epsilon-nets-original}. 

\begin{proof}[Proof of \cref{thm:vol-eps-nets}]
    Let $(v,\alpha)$ be one of the options presented in the theorem statement. We construct a weak $\epsilon$-net via a greedy algorithm, starting from $\mathcal S = \emptyset$. At each stage, if some collection $\mathcal F' \subseteq \mathcal F$ of size $\epsilon |\mathcal F|$ is not yet pierced by an ellipsoid in $\mathcal S$, use the appropraite volumetric selection lemma to choose an ellipsoid of volume $v$ that is contained in $\conv(\mathcal G)$ for at least $\beta_d\, \binom{|\mathcal F'|}{\alpha}$ choices of $\mathcal G \in \binom{\mathcal F'}{\alpha}$; then add this ellipsoid to the set $\mathcal S$. Every time an ellipsoid is added, the number of sets in $\binom{\mathcal F}{2d}$ whose convex hull is pierced by an ellipsoid in $\mathcal S$ increases by at least $\beta_d\, \binom{\epsilon |\mathcal F|}{\alpha}$. Therefore, this algorithm places at most
    \[
        \frac{\binom{|\mathcal F|}{\alpha}}{\beta_d\, \binom{\epsilon |\mathcal F|}{\alpha}}
        \leq c_d \epsilon^{-\alpha}
    \]
    ellipsoids in $\mathcal S$ before all $\epsilon|\mathcal F|$-tuples are pierced.
\end{proof}

\begin{remark}[Improving the piercing number for $v=d^{-d}$]
    For each set $F \in \mathcal F$, let $E_F$ be an ellipsoid of volume $d^{-d}$ contained in $F$. Using the parametrization of ellipsoids outlined at the beginning of this section, we obtain a point set $\{(A_F, x_F) : F \in \mathcal F\} \subseteq \R^{d(d+3)/2}$. Rubin's bound on weak $\epsilon$-nets \cite{Rubin-epsilon-net-upper-bound}, applied in $\R^{d(d+3)/2}$, shows that there is a set $P \subseteq \R^{d(d+3)/2}$ of $o_{\epsilon}(\epsilon^{d(d+3)/2 - 1/2})$ points (as long as $d \geq 2$) such $P \cap \conv(X) \neq \emptyset$ for any $X \subseteq \{x_F\}_{F \in \mathcal F}$ of size $|X| \geq \epsilon |\mathcal F|$. Because the set $S = \big\{ (A, x) : \vol\big(A(B^d) + x\big) \geq d^{-d} \big\}$ is convex,\footnote{This is because the determinant is log-concave on the set of positive-definite matrices.} the set $S \cap P$ is also a weak $\epsilon$-net and every point in $S \cap P$ corresponds to an ellipsoid in $\R^d$ of volume at least $d^{-d}$. This family of ellipsoids is a volumetric weak $\epsilon$-net for $\mathcal F$.
\end{remark}

\section{Homogeneous selection}\label{sec:hom-sel}

Our proof of \cref{thm:vol-hom-sel} draws inspiration from Pach's proof of the homogeneous point selection theorem \cite{homogeneous-selection}. To follow his approach, we will need to ``volumetrize'' several theorems in combinatorial geometry for which quantitative versions have not yet appeared, as well as prove a new volumetric Tverberg-type theorem. That is the work of \cref{sec:hom-sel_quant-lemmas}. In \cref{sec:hom-sel_proof}, we prove \cref{thm:vol-hom-sel}.

\subsection{More quantitative lemmas}\label{sec:hom-sel_quant-lemmas}

\subsubsection*{Tverberg's theorem with volume}

By parametrizing ellipsoids, Sarkar, Xue, and Sober\'on proved a colorful version of Tverberg's theorem with volumes:

\begin{theorem}[\!\!{\cite[Theorem 4.2.2]{pablo-concave-funcs}}]\label{thm:col-tv-vol}
    For any $\frac{1}{2} d(d+3)$ families $\mathcal E_1,\dots,\mathcal E_{\frac{1}{2} d(d+3)}$ in $\R^d$, each containing $2r-2$ ellipsoids of volume 1, there are $r$ disjoint transversals $\mathcal T_1,\dots,\mathcal T_r$, such that $\bigcap_{i=1}^r \conv(\mathcal T_i)$ contains an ellipsoid of volume 1.\footnote{The statement in \cite{pablo-concave-funcs} is slightly different, and only for the case when $r+1$ is a prime power. \cref{thm:col-tv-vol} follows from \cref{thm:col-tv-vol} by applying Bertrand's postulate. (See Corollary 2.4 in \cite{col-tv-opt} for details of this extension.)}
\end{theorem}

(The term ``disjoint transversals'' means that $\mathcal T_j$ contains exactly one set from each family $\mathcal E_i$, and $\mathcal T_{j_1} \cap \mathcal T_{j_2} = \emptyset$ whenever $j_1 \neq j_2$.) We could use \cref{thm:col-tv-vol} to prove a weaker version of the homogeneous selection theorem, with $\frac{1}{2}d(d+3)$ sets instead of $d+1$. To prove \cref{thm:vol-hom-sel}, we need to reduce the size of the transversals so each touches only $d+1$ families.

\begin{lemma}\label{thm:col-tv-vol-smaller}
    For any $\frac{1}{2} d(d+3)$ families $\mathcal F_1,\dots,\mathcal F_{d(d+3)/2}$ in $\R^d$, each containing $2\binom{\frac{1}{2} d(d+3)}{d+1}(r-1)$ sets of volume 1, there are $d+1$ families $\mathcal F_{j_1},\dots, \mathcal F_{j_{d+1}}$ and $r$ disjoint transversals $\mathcal T_1,\dots,\mathcal T_r$ of $\mathcal F_{j_1},\dots, \mathcal F_{j_{d+1}}$ such that $\bigcap_{i=1}^{r} \conv(\mathcal T_{i})$ contains an ellipsoid of volume $r^d d^{-d^2 (1 + o_d(1))}$.
\end{lemma}
\begin{proof}
    First, we apply John's theorem (\cref{thm:john}) to each set in $\mathcal F_i$ to get $\mathcal \mathcal E_i$. Then, \cref{thm:col-tv-vol} with $k \defeq \binom{d(d+3)/2}{d+1}(r-1)+1$ produces disjoint transversals $\mathcal S_1, \dots, \mathcal S_k$ of $\mathcal E_1, \dots, \mathcal E_{d(d+3)/2}$ such that $\bigcap_{i=1}^k \conv(\mathcal S_i)$ contains an ellipsoid $E$ of volume at least $d^{-d}$. From \cref{thm:refining-vol-tv}, there are subfamilies $\mathcal S_i' \subset \mathcal S_i$, each containing exactly $d+1$ sets, such that the volume of $\bigcap_{i=1}^k \mathcal S'_i$ is at least $(\frac{1}{5} d^4 k 2^{2d})^{-d}$. Each of the sets $S'_i$ is a transversal of some collection of $d+1$ of the sets $\mathcal E_1,\dots,\mathcal E_{d(d+3)/2}$. By the pigeonhole principle, there is a set of $r$ subfamilies $\mathcal S_{i_1}',\dots, \mathcal S_{i_r}'$ which are transversals of the same collection of $d+1$ sets, say $\mathcal E_{j_1},\dots,\mathcal E_{j_{d+1}}$.
    
    Each transversal $\mathcal S'_{i_a}$ of $\mathcal E_{j_1},\dots,\mathcal E_{j_{d+1}}$ corresponds to a transversal $\mathcal T_a$ of $\mathcal F_{j_1},\dots,\mathcal F_{j_{d+1}}$, and
    \[
        \vol\big(\bigcap_{a=1}^r \mathcal T_a\big)
        \geq \vol\big( \bigcap_{i=1}^k \mathcal S'_i \big)
        \geq \big( \frac{1}{5} d^4 k 2^{2d}\big)^{-d}
        \geq r^d d^{-d^2 (1 + o_d(1))}.\qedhere
    \]
\end{proof}

\subsubsection*{Same-type lemma}

% Our next result is a volumetric version of this theorem. To help distinguish the difference between the three types of sets used in the proof, convex sets are typeset as uppercase, like $F$; collections of convex sets are typeset in fancy script, like $\mathcal F$; and collections of sets of convex sets are typeset in \red{fancy script with an underline}, like $\set{H}$. We let $K^{(d+1)}_r (t)$ denote the complete $(d+1)$-uniform, $r$-partite hypergraph with $t$ vertices in each part.

The \emph{order type} of a collection of $d+1$ points $x_1,\dots,x_{d+1} \in \R^d$ is the sign of $\det(x_1-x_{d+1}, x_2 - x_{d+1}, \dots, x_d - x_{d+1})$. Intuitively, if $x_1,\dots,x_{d+1}$ do not lie on a hyperplane, the order type tells you on which side of the hyperplane defined by $x_1,\dots,x_d$ the point $x_{d+1}$ lies. (If the $d+1$ points \emph{do} lie on a single plane, the order type is $0$.) The order type of a larger collection of points $x_1, \dots, x_m \in \R^d$ is the labelled collection of the order types of all subsets of $d+1$ points.

\begin{definition}
    We say that a collection of sets $C_1,\dots, C_{m} \subset \R^{d+1}$ \emph{has an order type} if every list $(x_1,\dots,x_{m}) \in C_1 \times \cdots \times C_{m}$ has the same order type. In that case, the \emph{order type} of $(C_1,\dots, C_m)$ is defined to be the order type of any point $(x_1,\dots,x_m) \in C_1 \times \cdots \times C_{m}$.
\end{definition}

The order type of a point set is enough to determine which points lie in the convex hull of others. In other words, if $(x_0,\dots,x_m)$ and $(y_0,\dots,y_m)$ have the same order type, then $x_0 \in \conv(x_1,\dots,x_m)$ if and only if $y_0 \in \conv(y_1,\dots,y_m)$. The same statement is true when sets replace points.

The \emph{same-type lemma} refines a collection of point sets so that the resulting sets themselves have an order type. In other words,

\begin{theorem}
    For any finite point sets $X_1,\dots,X_m \subset \R^d$, there are point sets $Y_i \subseteq X_i$ such that
    \begin{itemize}
        \item $|Y_i| \geq c_{d,m} |X_i|$, and 
        \item every list $(y_1, \dots, y_m) \in Y_1 \times \cdots \times Y_m$ has the same order type.
    \end{itemize}
    The constant $c_{d,m}$ is positive and depends only on $d$ and $m$.
\end{theorem}

Our next goal is to state and prove a volumetric version of the same-type lemma.

\begin{theorem}[Same-type lemma with volume]\label{thm:same-type-volume}
    For any $m \geq d+1$ collections $\mathcal F_1, \dots, \mathcal F_m$ of volume-1 convex sets in $\R^d$, there are collections $\mathcal G_1,\dots, \mathcal G_m$ such that, for each $i \in [m]$,
    \begin{itemize}
        \item each $G \in \mathcal G_i$ corresponds to a unique $F \in \mathcal F_i$ such that $G \subseteq F$,
        \item $|\mathcal G_i| \geq \delta_{d,m}|\mathcal F_i|$,
        \item $\vol(G) \geq v_{d,m}$ for each $G \in \mathcal G_i$, and
        \item every collection $(G_1,\dots,G_m) \in \mathcal G_1,\dots,\mathcal G_m$ has an order type, and all of these order types are the same.
    \end{itemize}
    The constants $v_{d,m}$ and $\delta_{d,m}$ are positive and depend only on $d$ and $m$.
\end{theorem}

To prove \cref{thm:same-type-volume}, we will need the following lemma, used also to prove the same-type lemma for points, which gives a condition for $d+1$ convex sets to have an order type.

\begin{lemma}[Matou\v{s}ek \cite{matousek-lectures}, Lemma 9.3.2]\label{thm:hyperplane-for-same-type}
    If $C_1, \dots, C_{d+1}$ are convex sets in $\R^d$ such that $\bigcup_{i \in I} C_i$ and $\bigcup_{j \notin I} C_j$ can be separated by a hyperplane for every $\emptyset \subsetneq I \subsetneq [d+1]$, then $(C_1,\dots,C_{d+1})$ has an order type.
\end{lemma}
% \begin{proof}
%     If $C_1,\dots,C_{d+1}$ did not have an order type, there would be two sets of points, $\mathbf x = (x_i)_{i=1}^{d+1}$ and $\mathbf y = (y_i)_{i=1}^{d+1}$, in $\prod_{i=1}^{d+1} C_i$ with order types $+1$ and $-1$, respectively. The function $f\colon (\R^d)^{d+1} \to \R$ defined by 
%     \[
%         f(\mathbf z) = \det(z_2-z_1,\dots,z_{d+1}-z_1)
%     \]
%     is a continuous function of $\mathbf z$. In particular, taking $\mathbf z(\alpha) = \alpha \mathbf x + (1-\alpha)\mathbf y$, we have $f\big(\mathbf z(0)\big) < 0$ while $f\big(\mathbf z(1)\big) > 0$. For some value of $\alpha \in (0,1)$, we have $f\big(\mathbf z(\alpha)\big) = 0$. Then the $d+1$ points $\mathbf z(\alpha) = (z_1,\dots,z_{d+1})$ lie on a hyperplane; by Radon's theorem, there is a set $\emptyset \subsetneq I \subsetneq [d+1]$ such that $\conv(z_j : j \in I) \cap \conv(z_i : i \notin I) \neq \emptyset$. But that means that $\bigcup_{i \in I} C_i$ and $\bigcup_{j \notin I} C_j$ cannot be separated by a hyperplane.
% \end{proof}

\begin{proof}[Proof of \cref{thm:same-type-volume}]
    We will prove the result for $m=d+1$. Since the order type of an $m$-tuple with $m > d+1$ is a collection of order types of $(d+1)$-tuples, we can apply this case iteratively to each $(d+1)$-tuple of $\mathcal F_1,\dots, \mathcal F_m$; after $\binom{m}{d+1}$ applications, the resulting collections satisfy the conclusion of the lemma with $\delta_{d,m} \geq (\delta_{d,d+1})^{\binom{m}{d+1}}$ and $v_{d,m} \geq (v_{d,d+1})^{\binom{m}{d+1}}$.

    We will form $\mathcal G_i$ by iteratively refining $\mathcal F_i$. To get in the situation where each $(G_1,\dots,G_{d+1}) \in \prod_{i=1}^{d+1} \mathcal G_i$ has the same order type, we will construct $\mathcal G_i$ so that the sets $\conv(\mathcal G_1),\dots,\conv(\mathcal G_{d+1})$ satisfy the conditions of \cref{thm:hyperplane-for-same-type}. Since $I$ and $[d+1]\setminus I$ give the same partition of $[d+1]$, we need only consider half the nonempty proper subsets of $[d+1]$. We will take the sets that do not contain $d+1$. So, to start, we enumerate the subsets $\emptyset \subsetneq I \subsetneq [d]$ as $I_1,\dots,I_{2^{d}-1}$ and set $\mathcal G_i^0 = \mathcal F_i$. In the end, we will take $\mathcal G_i = \mathcal G_i^{2^d-1}$. 
    
    Here is how we construct $\mathcal G_i^{k+1}$ from $\mathcal G_i^k$. By induction, we will assume that each set in $\mathcal G_i^k$ has the same volume $\rho_{d,k} > 0$. Each $\mathcal G_i^k$ corresponds to a measure $\mu_i^k(A) = \sum_{G \in \mathcal G_i^k} \lambda (A \cap G)$, where $\lambda$ is the Lebesgue measure on $\R^d$. By the ham sandwich theorem (\cref{thm:ham-sandwich}), there is a hyperplane $H_k$ that simultaneously halves the $d$ measures $\mu_i^k$ for $1 \leq i \leq d$. Let $H^+_k$ and $H^-_k$ denote the two half-spaces determined by $H_k$, and define
    \[
        \mathcal G_i^{k\mathcal +} =
        \big\{G \cap H^+_k : \vol(G \cap H^+_k) \geq \frac{1}{3} \rho_{d,k} \big\}
    \]
    and similarly $\mathcal G_i^{k-}$. The collection $\mathcal G_i^{k+}$ contains at least $\frac{1}{4} |\mathcal G_i^k|$ sets, since otherwise
    \begin{equation}\label{eq:halving-plane}
        \frac{\vol(\mathcal G_i^k \cap H^+_k)}{\vol(\mathcal G_i^k)}
        = \frac{1}{|\mathcal G_i^k|} \sum_{G \in G_i^k} \frac{1}{\rho_{d,k}} \vol(G \cap H_k^+)
        < \frac{1}{4} \cdot 1 + \frac{3}{4} \cdot \frac{1}{3}
        = \frac{1}{2},
    \end{equation}
    implying that $H_k$ is not a halving hyperplane. Similarly, each $\mathcal G_i^{k-}$ contains at least $\frac{1}{4}|\mathcal G_i^k|$ sets.

    We define $\mathcal G_{d+1}^{k+}$ and $\mathcal G_{d+1}^{k-}$ similarly. Since $H_k$ does not halve $\mathcal G_{d+1}^k$, it is possible that one of these collections is small. However, the argument in \eqref{eq:halving-plane} shows that at least one of these two collections is large: if $\vol(\mathcal G_{d+1}^k \cap H^+_k) \geq 1/2$, then $|\mathcal G_{d+1}^{k+}| \geq \frac{1}{4} |\mathcal G_{d+1}^k|$; and if $\vol(\mathcal G_{d+1} \cap H^-) \geq 1/2$, then $|\mathcal G_{d+1}^{k-}| \geq \frac{1}{4} |\mathcal G_{d+1}^k|$.

    We now use $I_{k+1}$ to define $\mathcal G_i^{k+1}$:
    \begin{itemize}
        \item If $|\mathcal G_{d+1}^{k+}| \geq \frac{1}{4} |\mathcal G_{d+1}^k|$, define ${\tilde{\mathcal G}_i^{k+1}} = \mathcal G_i^{k+}$ for every $i \notin I_{k+1}$ and ${\tilde{\mathcal G}_i^{k+1}} = \mathcal G_i^{k-}$ for every $i \in I_{k+1}$.
        \item If $|\mathcal G_{d+1}^{k-}| \geq \frac{1}{4} |\mathcal G_{d+1}^k|$, define ${\tilde{\mathcal G}_i^{k+1}} = \mathcal G_i^{k-}$ for every $i \notin I_{k+1}$ and ${\tilde{\mathcal G}_i^{k+1}} = \mathcal G_i^{k+}$ for every $i \in I_{k+1}$.
    \end{itemize}
    To obtain $\mathcal G_i^{k+1}$, we replace each set $G \in \tilde{\mathcal G}_i^{k+1}$ with a convex set of volume exactly $\rho_{d,k+1} = \frac{1}{3}\rho_{d,k}$ contained inside it.

    Each set in $\mathcal G_i^{k+1}$ has the same volume $\rho_{d,k+1}$, and $|\mathcal G_i^{k+1}| \geq \frac{1}{4}|\mathcal G_i^k|$. Moreover, $\bigcup_{i \in I_{k+1}} \mathcal \conv(\mathcal G_i^{k+1})$ and $\bigcup_{j \notin I_{k+1}} \conv(\mathcal G_i^{k+1})$ are separated by the hyperplane $H_k$. Therefore the sets $\mathcal G_i \defeq \mathcal G_i^{2^d-1}$ satisfy the hypotheses of \cref{thm:hyperplane-for-same-type}, implying that the sets $\conv(\mathcal G_1), \dots, \conv(\mathcal G_{d+1})$ have an order type. As a consequence, every selection $(G_1,\dots, G_{d+1}) \in \prod_{i=1}^{d+1} \mathcal G_i$ has the same type.
\end{proof}

This proof gives $v_{d,d+1} \geq 3^{-2^d + 1}$ and $\delta_{d,d+1} \geq 4^{-2^d+1}$. These values aren't rigid: We may exchange a larger value of $v_{d,d+1}$ for a smaller value of $\delta_{d,d+1}$, and vice versa. If we redefine $\mathcal G_i^{k+}$ as 
\[
    \mathcal G_i^{k+} =
    \big\{G \cap H^+_k : \vol(G \cap H^+_k) \geq \alpha\, \rho_{d,k} \big\},
\]
then the same proof gives $v_{d,d+1} \geq \alpha^{2^d-1}$ and $\delta_{d,d+1} \geq (1-\frac{1}{2(1-\alpha)})^{2^d-1}$. In other words, if one of $v_{d,d+1}$ or $\delta_{d,d+1}$ is extraordinarily small, then the other may be nearly as large as $2^{-2^d + 1}$ (which is, admittedly, not so very large).

\subsection{Proof of the homogeneous selection theorem with volume}\label{sec:hom-sel_proof}

At last, we turn to the homogeneous selection theorem. We will prove a colorful version of the theorem, which is stronger than \cref{thm:vol-hom-sel}, and deduce \cref{thm:vol-hom-sel} from it.

\begin{theorem}[Colorful homogeneous selection with volume]\label{thm:colorful-hom-vol-sel}
    Suppose that $\mathcal F_1,\dots,\mathcal F_{\frac{1}{2} d(d+3)}$ are collections of volume-1 convex sets in $\R^d$ and $|\mathcal F_i| = n$ for every $i \in [d+1]$. If $n$ is sufficiently large, then there are
    \begin{itemize}
        \item $d+1$ of these families $\mathcal F_{i_1},\dots,\mathcal F_{i_{d+1}}$,
        \item a set $E$ with volume at least $v_d > 0$, and 
        \item collections $\widetilde{\mathcal F}_j \subseteq \mathcal F_{i_j}$
    \end{itemize}
    such that $|\widetilde{\mathcal F}_j|\geq c_d |\mathcal F_{i_j}|$ and $E \in \conv(\widetilde{F}_1,\dots,\widetilde{F}_{d+1})$ for any selection of sets $\widetilde{F}_i \in \widetilde{\mathcal F}_i$.
\end{theorem}
\begin{proof}
    Having prepared our mathematical \textit{mise en place}, the proof will follow a similar outline to Pach's homogeneous point selection theorem \cite{selection-thm-original}. One key difference between this proof and Pach's, besides the new volumetric results, is that we must refine the $\frac{1}{2}d(d+3)$ families to a selection of $d+1$ of them. In the first step of the proof, we choose these $d+1$ families $\mathcal F_{i_1},\dots, \mathcal F_{i_{d+1}}$ so that a positive fraction of their transversals intersect with large volume. After that, we use the weak hypergraph regularity lemma (\cref{thm:hypergraph-regularity}) and the same-type lemma to select the sets $\widetilde{\mathcal F}_i$.
    
    Let's start. We will use $K^{\ell}(t)$ to mean the complete $\ell$-partite, $\ell$-uniform hypergraph with $t$ vertices in each part.\footnote{In other words, there are $\ell$ vertex sets $V_1,\dots, V_\ell$, each with $t$ vertices, and the edges of $K^{\ell}(t)$ are all those sets that have exactly one vertex in each of the $V_i$.} In our case, we will take $\ell = \frac{1}{2}d(d+3)$ and $t = 2\binom{\frac{1}{2} d(d+3)}{d+1}(r-1)$. Consider the complete $\ell$-partite hypergraph whose edges are tuples $(F_1,\dots,F_\ell) \in \prod_{i=1}^{\ell} \mathcal F_i$, which contains $\binom{n}{t}^\ell \geq \gamma_d n^{\ell t}$ copies of $K^{\ell}(t)$. Each instance of $K^{\ell}(t)$ corresponds to $\ell$ subfamilies $\mathcal K_{1} \subseteq \mathcal F_{i_1}, \dots, \mathcal K_{\ell} \subseteq \mathcal F_{i_{\ell}}$, each containing exactly $t$ sets. The reduced volumetric colorful Tverberg theorem (\cref{thm:col-tv-vol-smaller}) implies that among these, there are $d+1$ families $\mathcal K_{j_1}, \dots, \mathcal K_{j_{d+1}}$ and $d+1$ disjoint transversals $\mathcal T_1,\dots,\mathcal T_{d+1}$ of these families, such that $\bigcap_{i=1}^{d+1} \conv(\mathcal T_i)$ has volume at least $v_d' \defeq r^d d^{-d^2 ( 1 + o_d(1))}$.

    Each tuple of transversals $(\mathcal T_1,\dots,\mathcal T_{d+1})$ is contained in at most $n^{\ell t - (d+1)^2}$ different appearances of $K^{\ell}(t)$ in $\prod_{i=1}^{d+1} \mathcal F_i$, so in total, we have at least
    \[
        \frac{\gamma_d n^{\ell t}}{n^{\ell t - (d+1)^2}}
        = \gamma_d\, n^{(d+1)^2}
    \]
    \emph{distinct} disjoint transversals such that $\bigcap_{i=1}^{d+1} \conv(\mathcal T_i)$ has volume at least $v_d'$. By the pigeonhole principle, there are $d+1$ families $\mathcal F_{i_1}, \dots, \mathcal F_{i_{d+1}}$ that together contain at least $\gamma_d \binom{\ell}{d+1}^{-1} n^{(d+1)^2}$ of these transversals. For ease of notation, we will reorder the $\mathcal F_i$ so that $(\mathcal F_{i_1},\dots,\mathcal F_{i_{d+1}}) = (\mathcal F_1,\dots,\mathcal F_{d+1})$. We will only need these $d+1$ families from this point on.
    
    Let
    \[
        \mathcal C \defeq \big\{\!\conv(C_1, \dots, C_{d+1}) : C_i \in \mathcal F_{i} \big\}.
    \]
    Each tuple of transversals we found above corresponds to a different collection of $d+1$ sets in $\mathcal C$ whose intersection has volume at least $v_d'$. Since $|\mathcal C| = n^{d+1}$, we know that $\mathcal C$ contains at least $\gamma_d \binom{\ell}{d+1}^{-1} n^{(d+1)^2} \geq \gamma'_d \binom{|\mathcal C|}{d+1}$ such intersecting $(d+1)$-tuples. The volumetric fractional Helly theorem (\cref{thm:frac-helly-vol-d+1}) then tells us that there is a set $\set{H} \subseteq \mathcal C$ of size $\lvert \set{H} \rvert \geq \beta_{d}\, n^{d+1}$ whose intersection has volume at least $v_d'' > 0$. We will call this intersection $E$.
    
    Now for the second step. Because each set in $\mathcal C$ comes from the convex hull of a transversal of $\mathcal F_1, \dots, \mathcal F_{d+1}$, we may also consider $\set{H}$ to be a subhypergraph of $\prod_{i=1}^{d+1} \mathcal F_i$ with at least $\beta_d\, n^{d+1}$ edges. If $n$ is sufficiently large, then the hypergraph regularity lemma (\cref{thm:hypergraph-regularity}) guarantees subsets $\mathcal F_i' \subseteq \mathcal F_i$ such that
    \begin{itemize}
        \item $|\mathcal F_i'| \geq \alpha_d |\mathcal F_i|$ and
        \item whenever $\mathcal C_i \subseteq \mathcal F_i'$ with $|\mathcal C_i| \geq \delta_{d,d+2} |\mathcal F_i'|$ for each $i \in [d+1]$, at least one edge of $\set{H}$ lies in $\prod_{i=1}^{d+1} \mathcal C_i$.
    \end{itemize}
    Here, $\delta_{d,d+2}$ is the constant that appears in the volumetric same-type theorem (\cref{thm:same-type-volume}).
    
    Now apply \cref{thm:same-type-volume} to the $d+2$ families $(\mathcal F_1',\dots,\mathcal F_{d+1}',\{E\})$ to get subfamilies $(\mathcal G_1,\dots,\mathcal G_{d+1},\{\widetilde E\})$ with same-type transversals. Every selection $(G_1,\dots,G_{d+1}) \in \prod_{i=1}^{d+1} \mathcal G_i$ corresponds to a unique tuple $(F'_1,\dots, F'_{d+1}) \in \prod_{i=1}^{d+1} \mathcal F'_i$ such that $G_i \subseteq F'_i$. Because $|\mathcal G_i| \geq \delta_{d,d+2} |\mathcal F_i'|$, there is one selection $(G^*_1,\dots,G^*_{d+1})$ such that the corresponding tuple $(F^*_1,\dots,F^*_{d+1})$ lies in $\set{H}$. That implies that $E \subseteq \conv(F^*_1,\dots,F^*_{d+1})$. Since the order type is maintained by taking subsets, we conclude that $\widetilde E \subseteq \conv(G^*_1,\dots,G^*_{d+1})$.
    
    But! All transversals in $(\mathcal G_1,\dots,\mathcal G_{d+1},\{\widetilde E\})$ have the same order type, so $\widetilde E \subseteq \conv(G_1,\dots,G_{d+1})$ for \emph{every} $(G_1,\dots,G_{d+1}) \in \prod_{i=1}^{d+1} \mathcal G_i$. And now we simply set
    \[
        \widetilde{\mathcal F}_i
        = \big\{ F \in \mathcal F'_i : F \supseteq G \text{ for some } G \in \mathcal G_i \big\}.
    \]
    Then $\widetilde E \subseteq \conv(\widetilde F_1, \dots, \widetilde F_{d+1})$ for any choice of $\widetilde F_i \in \widetilde{\mathcal F}_i$.
    So the theorem holds with $v_d = \vol(\widetilde E)$ and $c_d = \delta_{d,d+2}$.
\end{proof}

\begin{proof}[Proof of \cref{thm:vol-hom-sel}]
    Let $\mathcal F_1,\dots,\mathcal F_{\frac{1}{2}d(d+3)}$ be disjoint subfamilies of $\mathcal F$, each with exactly $\big\lfloor 2|\mathcal F|/d(d+3) \big\rfloor$ sets, and apply \cref{thm:colorful-hom-vol-sel}.
\end{proof}

\section{Further quantitative selection theorems}\label{sec:quant-extensions}

This section illustrates the flexibility of this paper's proofs by extending the results to diameter and surface area. \cref{sec:selection-theorem-extensions} addresses selection and weak $\epsilon$-net theorems and \cref{sec:hom-sel-extensions} addresses the homogeneous sleection theorem.

\subsection{Selection theorems}\label{sec:selection-theorem-extensions}

The proof of \cref{thm:selection-d+1} is quite modular. It has four main ingredients:
\begin{itemize}
    \item the quantitative fractional Helly theorem (\cref{thm:vol-frac-helly}),
    \item the quantitative Tverberg theorem (\cref{thm:ellipse-tverberg}),
    \item a quantitative finite-approximation theorem (\cref{thm:quant-steinitz}),\footnote{A theorem of the form: If $\conv(X)$ has volume 1, then there are at most $k_d$ points in $X$ whose convex hull has volume at least $v_d > 0$.} and
    \item the fact that $\sum_{i} \vol(C_i) \geq \vol\big(\bigcup_i C_i)$ (used in \cref{thm:vol-planes}).
\end{itemize}
To prove a version of \cref{thm:selection-d+1} for different quantitative functions, we only need to prove the analogues of these four statements and then copy the proof of \cref{thm:selection-d+1}, replacing statements for volumes with the corresponding statements for the other quantitative function. In this way, we can obtain a variety of other quantitative selection theorems. Moreover, once we have a selection theorem, we get a weak $\epsilon$-net theorem essentially for free, using the same greedy algorithm that proves \cref{thm:vol-eps-nets}.

Take, for example, diameter as the new quantitative function. In this case, all the necessary ingredients are ripe for the plucking from the fertile field of mathematical literature: the fractional Helly theorem for diameter appears in \cite{melange-diameter, soberon-diameter-helly}; the diameter Tverberg theorem in \cite{pablo-concave-funcs}; and the finite-approximation theorem in \cite{discrete-quant-tverberg} (Lemma 3.1, which implies that any segment contained in $\conv(X)$ is contained in the convex hull of at most $2d$ points from $X$). Sliding these results into the framework of \cref{thm:selection-d+1}'s proof, we obtain:

\begin{theorem}[Selection theorem for diameter]\label{thm:selection-d+1-diam}
    For any finite family $\mathcal F$ of convex sets in $\R^d$ with diameter 1, there is a set $E$ with diameter at least $4^{-d^2 (1 + o_d(1) )}$ that is contained in $\conv(\mathcal G)$ for at least $c_d \binom{|\mathcal F|}{d+1}$ choices of $\mathcal G \in \binom{\mathcal F}{d+1}$.
\end{theorem}

If we feed this selection theorem into the greedy algorithm from the proof of \cref{thm:vol-eps-nets}, we get a weak $\epsilon$-net theorem free of charge:

\begin{theorem}[Weak $\epsilon$-nets for diameter]
    For every finite family $\mathcal F$ of diameter-1 convex sets and every $\epsilon > 0$, there is a family $\mathcal S$ of $O_d(\epsilon^{-(d+1)})$ sets, each of diameter at least $\nu_d > 0$, such that $\conv(\mathcal G)$ contains a set in $\mathcal S$ whenever $\mathcal G$ contains at least $\epsilon |\mathcal F|$ sets in $\mathcal F$.
\end{theorem}

But diameter, while the most commonly studied quantitative parameter besides volume, is not the only parameter to which this method applies. Rolnick and Sober\'on's combinatorial results \cite{quant-p-q-theorems} include fractional Helly, Tverberg, and finite-approximation theorems for surface area, as well. Fitting those results into the framework of our proof gives:

\begin{theorem}[Selection theorem for surface area]\label{thm:selection-d+1-sa}
    For any finite family $\mathcal F$ of convex sets in $\R^d$ with surface area 1, there is a set $E$ with surface area at least $\textup{sa}_d>0$ that is contained in $\conv(\mathcal G)$ for at least $c_d \binom{|\mathcal F|}{d+1}$ choices of $\mathcal G \in \binom{\mathcal F}{d+1}$.
\end{theorem}

\begin{theorem}[Weak $\epsilon$-nets for surface area]
    For every finite family $\mathcal F$ of diameter-1 convex sets and every $\epsilon > 0$, there is a family $\mathcal S$ of $O_d(\epsilon^{-(d+1)})$ sets, each of surface area at least $\textup{sa}_d > 0$, such that $\conv(\mathcal G)$ contains a set in $\mathcal S$ whenever $\mathcal G$ contains at least $\epsilon |\mathcal F|$ sets in $\mathcal F$.
\end{theorem}

While Rolnick and Sober\'on also prove selection and weak $\epsilon$-net results in \cite{quant-p-q-theorems}, they focus on maximizing the guaranteed surface area, while our aim is to minimize the size of the net.

\subsection{Homogeneous selection theorem for diameter}\label{sec:hom-sel-extensions}

The main hurdles in proving a homogeneous selection theorem for other parameters the colorful Tverberg and a fractional Helly theorems. The other results utilized in our proof (such as \cref{thm:vol-planes} and the same-type lemma) depend on properties that hold for most continuous geometric parameters, so don't require anything more than cosmetic adjustment.

Happily, Frankl, Jung, and Tomon have already proven a diameter fractional Helly theorem for $(d+1)$-tuples \cite{improved-quant-frac-helly}, so we only need a colorful Tverberg theorem. A few papers \cite{soberon-diameter-helly,pablo-concave-funcs} have proven \emph{monochromatic} Tverberg theorems for diameter, and the best of these requires at least $8d^2(r-1)+1$ segments to form $r$ intersecting sets. However, no colorful Tverberg theorems have yet appeared. Our next theorem not only provides a colorful Tverberg theorem for diameter, but also it requires fewer total sets, by a factor of $d$.

\begin{theorem}[Colorful Tverberg for diameter]\label{thm:col-tv-diam}
    For any $2d$ families $\mathcal F_1,\dots,\mathcal F_{2d}$ in $\R^d$, each containing $4r-4$ line segments of length 1, there are $r$ disjoint transversals $\mathcal T_1,\dots,\mathcal T_r$ such that $\bigcap_{i=1}^r \conv(\mathcal T_i)$ contains a line segment of length at least $\delta_d > 0$.
\end{theorem}
\begin{proof}
    Given a unit vector $v$, the \emph{$v$-width} of a convex set $C$ is $\max_{x,y\in C}\, \langle v, y-x\rangle$. To prove the theorem, we will find subfamilies $\mathcal F_i' \subseteq \mathcal F_i$ and a direction vector $v$ such that the $v$-width of every segment in $\mathcal F'_i$ is at least $\delta_d$. Then we can apply the usual colorful Tverberg theorem in a higher-dimensional space, using a parametrization argument similar to our proof of \cref{thm:selection-d(d+3)/2}. ($v$-width is a common tool for diameter theorems in combinatorial geometry; see \cite{soberon-diameter-helly,melange-diameter,pablo-concave-funcs} for other applications.)
    
    Let $\mu$ be the surface measure of the sphere $S^{d-1}$, normalized so that $\mu(S^{d-1}) = 1$. Using any fixed unit vector $x \in S^{d-1}$, we define $\delta_d$ as the real number such that
    \[
        \mu\big( \{ v \in S^{d-1} : |\langle v,x\rangle| > \delta_d\} \big) = 1 - \frac{1}{5d}.
    \]
    The \emph{symmetric cap} about $x \in S^{d-1}$ will be 
    \[
        C(x) = \{ v \in S^{d-1} : |\langle v,x\rangle| > \delta_d\}.
    \]

    We use $[x,y]$ to denote the line segment between the points $x,y \in \R^d$. Consider the collection $\mathcal C_i \defeq \{C(y-x) : [x,y] \in \mathcal F_i\}$. We claim that the measure of points in $S^{d-1}$ that are covered by at most $|\mathcal F_i|/2$ symmetric caps in $\mathcal C_i$ is less than $1/2d$. If it were not, we get a contradiction: Counting multiplicity, we have $\mu(\bigcup \mathcal C_i) = (1-1/5d) |\mathcal F_i|$ for each $i \in [2d]$; however,
    \[
        \mu({\textstyle \bigcup{}} \mathcal C_i)
        < \frac{1}{2d} \cdot \frac{1}{2} |\mathcal F_i| + \frac{2d-1}{2d} \cdot |\mathcal F_i|
        = \Big( 1- \frac{1}{4d}\Big) |\mathcal F_i|.
    \]
    
    As a consequence, there is a single vector $v \in S^{d-1}$ that is simultaneously covered by at least $|\mathcal F_i|/2$ caps in each $\mathcal C_i$. We define
    \[
        \mathcal F'_i \defeq \{ [x,y] \in \mathcal F_i : v \in C(y-x) \}.
    \]
    Each line segment in $\mathcal F'_i$ has $v$-width at least $\delta_d$.

    We now move to the parametrization argument. By reducing the lengths, we may assume that each segment in the $\mathcal F'_i$ has $v$-width exactly $\delta_d$. We map the segments in $\mathcal F_i$ to a point set $X_i \subseteq \R^{2d}$ by $\phi\colon [x,y] \mapsto (x,y)$. Since $\langle y-x,v\rangle = \delta_d$ for every $[x,y] \in \bigcup_i \mathcal F'_i$, the points lie in a $(2d-1)$-dimensional subspace of $\R^{2d}$. Since $|X_i| \geq |\mathcal F_i|/2 = 2r-2$, we may apply the colorful Tverberg theorem \cite{colorful-tverberg-optimal} to the set $X_1,\dots, X_{2d}$ to obtain disjoint transversals $\mathcal S_1, \dots, \mathcal S_{r}$ such that $\bigcap_{i=1}^{2d} \conv(\mathcal S_i)$ contains a point $(\bar x, \bar y)$ such that $\langle \bar y - \bar x, v\rangle = \delta_d$. Each transversal $\mathcal S_i$ corresponds to a transversal $\mathcal T_i \defeq \{ [x,y] : (x,y) \in \mathcal S_i\}$ of $\mathcal F'_i$, and $[\bar x, \bar y] \subseteq \bigcup_{i=1}^r \conv(\mathcal T_i)$. Since $[\bar x, \bar y]$ has $v$-width, and thus length, at least $\delta_d$, we are done.
\end{proof}

\begin{remark}
    The argument in the previous proof that produces the vector $v$ can be rephrased probabilistically. Let $u$ be a uniform random vector in $S^{d-1}$. By linearity of expectation,
    \[
        \E \big[ \#\{C \in \mathcal C_i : u \notin C\}\big]
        = \frac{1}{5d} |\mathcal F_i|.
    \]
    Markov's inequality implies that
    \[
        \mathbb{P}\big( \#\{C \in \mathcal C_i : u \notin C\} \geq \frac{1}{2}|\mathcal F_i| \big)
        \leq \frac{|\mathcal F_i|/5d}{|\mathcal F_i|/2}
        = \frac{2}{5d},
    \]
    and the union bound tells us that
    \[
        \mathbb{P}\big( \#\{C \in \mathcal C_i : u \notin C\} \geq \frac{1}{2}|\mathcal F_i| \text{ for some } i \in [2d]\big)
        \leq \frac{4d}{5d}
        < 1,
    \]
    so there is some vector $v$ that is simultaneously contained in at least $|\mathcal F_i|/2$ caps of $\mathcal C_i$, for each $i \in [2d]$.\hfill$\lozenge$
\end{remark}

In fact, the number $4r-4$ can be replaced by $c(2r-2)$ for any $c > 1$, but smaller values of $c$ result in smaller values of $\delta_d$.

With the pieces in hand, proving a homogeneous selection theorem is a matter of substitution.

\begin{theorem}[Colorful homogeneous selection with diameter]\label{thm:colorful-diam-hom-sel}
    Suppose that $\mathcal F_1,\dots,\mathcal F_{2d}$ are collections of diameter-1 convex sets in $\R^d$ and $|\mathcal F_i| = n$ for every $i \in [d+1]$. If $n$ is sufficiently large, then there are
    \begin{itemize}
        \item $d+1$ of these families $\mathcal F_{i_1},\dots,\mathcal F_{i_{d+1}}$,
        \item a set $E$ with diameter at least $\delta_d > 0$, and 
        \item collections $\widetilde{\mathcal F}_j \subseteq \mathcal F_{i_j}$
    \end{itemize}
    such that $|\widetilde{\mathcal F}_j|\geq c_d |\mathcal F_{i_j}|$ and $E \in \conv(\widetilde{F}_1,\dots,\widetilde{F}_{d+1})$ for any selection of sets $\widetilde{F}_i \in \widetilde{\mathcal F}_i$.
\end{theorem}\closeprooftrue
\begin{proof}[Proof sketch]
    Nearly identical to the proof of \cref{thm:colorful-hom-vol-sel}, except the volumetric colorful Tverberg theorem, fractional Helly theorem, and same-type lemma are replaced by the corresponding results for diameter. (The same-type theorem for diameter may be proven using the exact same proof as for volume, and \cite{improved-quant-frac-helly} proves the diameter fractional Helly theorem for $(d+1)$-tuples.)
\end{proof}

\begin{corollary}[Homogeneous selection with diameter]
    For any large enough family $\mathcal F$ of diameter-1 convex sets in $\R^d$, there are $d+1$ disjoint subfamilies $\mathcal G_1, \dots, \mathcal G_{d+1}$ of $\mathcal F$, each containing at least $c_d |\mathcal F|$ sets, and a convex set $E \subseteq \R^d$ of diameter $\delta_d > 0$, such that $E \subseteq \conv(G_1,\dots,G_{d+1})$ for every selection $G_i \in \mathcal G_i$.
\end{corollary}
\begin{proof}
    Let $\mathcal F_1,\dots,\mathcal F_{2d}$ be disjoint subsets of $\mathcal F$, each with exactly $\lfloor |\mathcal F|/2d\rfloor$ sets, and apply \cref{thm:colorful-diam-hom-sel}.
\end{proof}

\section{Open problems and future work}

Quantitative colorful Tverberg theorems, which are central to our proof of the homogeneous selection theorem, have so far received little attention. One basic question, much in line with previous research on colorful theorems, is the minimum number of color classes necessary for such a theorem. Many quantitative theorems hold with conditions on $2d$-tuples, so it's plausible the same is true for the colorful Tverberg theorem.

\begin{conjecture}\label{conj:tverberg}
    For any $2d$ families $\mathcal F_1,\dots,\mathcal F_{2d}$ in $\R^d$, each containing $f(d,r)$ convex sets of volume 1, there are $r$ disjoint transversals $\mathcal T_1,\dots,\mathcal T_r$, such that $\bigcap_{i=1}^r \conv(\mathcal T_i)$ has volume at least $v_d > 0$.
\end{conjecture}

Any proof of a colorful Tverberg theorem with fewer color classes immediately implies a colorful homogeneous selection theorem with the same number of color classes. Given that the volumetric (nonhomogeneous) selection theorem (\cref{thm:selection-d+1}) and the fractional Helly theorem (\cref{thm:frac-helly-vol-d+1}) hold for $(d+1)$-tuples, it may even be the case that the colorful Tverberg theorem holds with $d+1$ color classes

An alternative goal to decreasing the number of color classes is maximizing the volume $v_d$, in the style of Rolnick and Sober\'on's quantitative results \cite{quant-p-q-theorems}. What is the minimum number of color classes, or the minimum number of convex sets in each class, needed to guarantee a volume of $1-\delta$? Alternatively, in the monochromatic version (\cref{thm:vol-hom-sel}), where the number of partitions is optimal, how large can the guaranteed volume be? Similarly, how much can the volume guarantee in \cref{thm:same-type-volume} be increased?

Then there is the colorful homogeneous selection theorem itself. Do we really need $\frac{1}{2}d(d+3)$ color classes? Probably not.

\begin{conjecture}[Optimal colorful homogeneous selection with volume]\label{conj:colorful-hom-vol-sel}
    Suppose that $\mathcal F_1,\dots,\mathcal F_{d+1}$ are collections of volume-1 convex sets in $\R^d$ and $|\mathcal F_i| = n$ for every $i \in [d+1]$. If $n$ is sufficiently large, then there are
    \begin{itemize}
        \item collections $\widetilde{\mathcal F}_i \subseteq \mathcal F_i$ and 
        \item a set $E$ with volume at least $v_d > 0$
    \end{itemize}
    such that $|\widetilde{ \mathcal F}_i|\geq c_d |\mathcal F_i|$ and $E \in \conv(\widetilde{F}_1,\dots,\widetilde{F}_{d+1})$ for any selection of sets $\widetilde{F}_i \in \widetilde{\mathcal F}_i$.
\end{conjecture}

Improving the colorful Tverberg theorem is one approach to this conjecture, but there may be others, as well.

Of course, \Cref{conj:colorful-hom-vol-sel,conj:tverberg} may be phrased for diameter or surface area, as well. In fact, another good problem is to prove a colorful Tverberg theorem for surface area, with any number of color classes, as no such result has yet appeared. Rolnick and Sober\'on have proven a fractional Helly theorem for surface area \cite{quant-p-q-theorems}. Put together, these two results would prove a homogeneous selection theorem for surface area (though the number of color classes may be fairly large without further improvements).

As for the first half of this paper, the most obvious question is reducing the size of a volumetric weak $\epsilon$-net. We'll say that a family $\mathcal F$ \emph{has a weak $\epsilon$-net of $k$ sets and volume $v > 0$} if there is a collection $\mathcal S$ of $k$ sets, each of volume $v$, such that $\conv(\mathcal G)$ contains a set in $\mathcal S$ whenever $|\mathcal G|\geq \epsilon \mathcal F$.

\begin{question}\label{q:optimal-e-nets}
     What is the minimum value $\alpha_d$ such that there is a real number $v_d > 0$ such that every finite family $\mathcal F$ of convex sets in $\R^d$ has a weak $\epsilon$-net with $O_d(\epsilon^{-\alpha_d})$ sets and volume $v_d$?
\end{question}

Finding small nets is a hard problem even for point sets, so it may seem unlikely that adding volume to the problem will make it any easier. But one never knows, and perhaps working with sets will allow some freedom that points do not. In any case, while determining the optimal $\alpha_d$ may be a difficult problem, obtaining a slight improvement, as has been done for weak $\epsilon$-nets of points \cite{eps-net-upper-d,matousek-eps-net-upper,matousek-eps-net-survey}, may be more tractable.

\begin{question}
    (Regarding \cref{q:optimal-e-nets}) Is $\alpha_d < d+1$?
\end{question}

Similarly, what of lower bounds for $\alpha_d$? Here is a simple construction to show that $\alpha_d \geq 1$. Let $H_0,\dots,H_{1/\epsilon}$ be parallel hyperplanes in $\R^d$;  for each $i \in [1/\epsilon]$, let $\mathcal F_i$ be a collection of $\epsilon n$ volume-1 sets that lie between $H_{i-1}$ and $H_i$; and define $\mathcal F = \bigcup_{i=1}^{1/\epsilon} \mathcal F_i$. Since $|\mathcal F_i| = \epsilon |\mathcal F|$, any $\epsilon$-net for $\mathcal F$ must contain at least one set between $H_{i-1}$ and $H_{i}$. So any $\epsilon$-net contains at least $\epsilon^{-1}$ sets.

Finally, a number of the results in this paper rely on \cref{thm:vol-planes} to reduce to a $(d+1)$-tuple from a larger set. The volume in the conclusion of this theorem is likely not close to optimal; how much can it be improved?

% \red{discrete containment} -- \textcolor{purple}{$O_d(\epsilon^{kd})$ for $k \geq 2$, no loss. Is there a smaller $\epsilon$-net, with fewer points (like $k/c_d$ or smth?) Would need a result that says for any set $X$ w/ $|\conv(X) \cap S| \geq k$, there is a set $Y \subseteq X$ with $|Y| \leq f(d)$ such that $|Y \cap S| \geq k/c_d$. This is not possible for all discrete $S$ (take $S=$ a big convex polygon, then $|Y \cap S| \leq |Y|$). But if $S = \Z^d$, is it true?}

\vspace{1.5\baselineskip}
{\centering
    \textsc{acknowledgements}\\[0.3em]
}
\addcontentsline{toc}{section}{Acknowledgements}
\noindent
Thanks to Pablo Sober\'on for providing helpful pointers on the literature and the National Science Foundation for partially supporting this work through a Graduate Research Fellowship (Grant No. 2141064).

\vspace{1.5\baselineskip}
\noindent
\textsc{Travis Dillon} (\texttt{dillont@mit.edu}), Massachusetts Institute of Technology: Cambridge, MA, USA

\newpage

\let\OLDthebibliography\thebibliography
\renewcommand\thebibliography[1]{
  \OLDthebibliography{#1}
  \setlength{\parskip}{0pt}
  \setlength{\itemsep}{3pt plus 0.3ex}
}

{\setstretch{1.0}
    \bibliographystyle{amsplain-nodash}
    \bibliography{bibliography}
}

\newpage

\appendix
\section{Gallery of theorems}\label{sec:misc}
\setcounter{theorem}{0}

\subsection*{Quantitative combinatorial geometry}

\begin{theorem}[Sarkar--Xue--Sober\'on \cite{pablo-concave-funcs}, Theorem 4.1.2]\label{thm:ellipse-tverberg}
    Any family $\mathcal E$ of $(r-1)\big(\frac{1}{2}d(d+3)+1)+1$ ellipsoids in $\R^d$ with volume 1 can be partitioned into $r$ disjoint families $\mathcal E_1, \dots, \mathcal E_r$ such that $\bigcap_{i=1}^r \conv(\mathcal E_i)$ contains an ellipsoid of volume 1.
\end{theorem}

\begin{theorem}[Stenitz's theorem with volume; Ivanov--Nasz\'odi \cite{quant-steinitz}, Theorem 3]\label{thm:quant-steinitz}
    Any set $X \subseteq \R^d$ whose convex hull contains the unit ball $B^d$ has a subset $Y \subseteq X$ of at most $2d$ points such that $\conv(Y) \supseteq (5d^2)^{-1} B^d$.
\end{theorem}

\begin{theorem}[Fractional Helly with volume; Frankl--Jung--Tomon \cite{improved-quant-frac-helly}, Theorem 1.2]\label{thm:frac-helly-vol-d+1}
    For every positive integer $d$ and  $\alpha \in (0,1)$, there exist constants $v_d > 0$ and $\beta_{d,\alpha} > 0$ such that the following holds: For any family $\mathcal F$ of convex sets in $\R^d$ such that at least $\alpha \binom{|\mathcal F|}{d+1}$ of the $(d + 1)$-tuples of members of $\mathcal F$ have an intersection whose volume is at least 1, there is a subfamily $\mathcal G \subseteq \mathcal F$ with at least $\beta_{d,\alpha}$ sets such that $\bigcap \mathcal G$ has volume at least $v_d$.
\end{theorem}

\begin{theorem}[Ham sandwich theorem; Steinhaus \cite{ham-sandwich}]\label{thm:ham-sandwich}
    If $\mu_1,\dots,\mu_d$ are finite measures in $\R^d$ that assign measure 0 to every hyperplane, then there is a hyperplane whose half-spaces both have measure $1/2$ according to each of the measures $\mu_1,\dots,\mu_d$.
\end{theorem}

\subsection*{Hypergraph structure}

\begin{theorem}[Weak regularity lemma for hypergraphs; Pach \cite{homogeneous-selection}, Theorem 2.3]\label{thm:hypergraph-regularity}
    Let $\mathcal H$ be a $k$-partite $k$-uniform hypergraph with vertex partition $X_1,\dots, X_k$, such that $|X_i| = n$ for every $1\leq i \leq k$. Suppose also that $\mathcal H$ contains at least $\beta n^k$ edges and fix an $\epsilon \in (0,1/2)$. For sufficiently large $n$ (in terms of $k$, $\beta$, and $\epsilon$), there are equal-size subsets $Y_i \subseteq X_i$ such that
    \begin{itemize}
        \item $|Y_i| \geq \beta^{1/\epsilon^k} |X_i|$,
        \item the subhypergraph induced on $\prod_{i=1}^k Y_i$ contains at least $\beta |Y_i|^k$ edges, and
        \item any subhypergraph $(Z_1, \dots, Z_k) \subseteq \prod_{i=1}^k Y_i$ with $|Z_i| \geq \epsilon |Y_i|$ contains at least 1 edge of $\mathcal H$.
    \end{itemize}
\end{theorem}

% \begin{theorem}[Hypergraph supersaturation; Erd\H{o}s--Simonovits \cite{hypergraph-supersaturation}, Corollary 2]\label{thm:supersaturation}
%     If $\set H$ is a $(d+1)$-uniform hypergraph on $n$ vertices with $\alpha \binom{n}{d+1}$ edges, then $\set H$ contains at least $c_{d,t} \alpha^{t^{d+1}} n^{(d+1)t}$ copies of $K^{d+1}(t)$.
% \end{theorem}

\end{document}

%% file: header.tex
\usepackage[T1]{fontenc}

\usepackage[letterpaper, portrait, left=1.1in, right=1.1in, top=1.25in, bottom=1.25in, footnotesep=1.5\baselineskip]{geometry}
\usepackage[nocompress]{cite} % sorts multiple citations
\usepackage[dvipsnames]{xcolor}
\usepackage{contour}
    \contourlength{0.5pt}
\usepackage{tocloft}

\usepackage{environ} % for wrapping environments in macros
\usepackage[framemethod=tikz]{mdframed}
\usepackage{wrapfig}
\usepackage{mathtools}
    \usetikzlibrary{calc, math, cd, arrows.meta, braids, decorations.pathreplacing, decorations.markings, decorations.pathmorphing, bending, shapes}
    \usepackage{tikz-3dplot}
    \tikzset{vertex/.style={draw, shape=circle, inner sep=1.5pt, minimum size=4pt}}
    \tikzset{<->/.tip={Latex}}
    \tikzset{shorten > = 2pt, shorten <=2pt}
    \tikzset{smallnode/.style={every node/.style={draw, fill=black, shape=circle, inner sep=0pt, minimum size=4pt}, scale=0.7}}
    \tikzset{drawnode/.style={fill,shape=circle,inner sep=0pt, minimum size=3pt}}
\usepackage{pgfornament}
\usepackage{fourier-orns} % for warning symbol
\usepackage{stmaryrd}
\usepackage{halloweenmath}
\usepackage[shortlabels]{enumitem}
    \newlength{\circlabelwidth}
    \setlist{nosep}
    \setlist[enumerate]{label=\textup{\arabic*.}}
    \newlist{subprob}{enumerate}{2}
        \setlist[subprob,1]{label={(\roman*)}}
        \setlist[subprob,2]{label={(\arabic*)}}
    \setlist[itemize]{labelindent=10pt,labelwidth=\circlabelwidth,leftmargin=!,label=$\circ$}
    \newlist{problems}{enumerate}{3}
        \setlist[problems,1]{before=\setupstar,label=\textup{\arabic*.}, itemsep=2pt, topsep=8pt,ref=\textup{\arabic*}}
        \setlist[problems,2]{before=\setupstar,label=(\alph*),parsep=0pt}
        \setlist[problems,3]{before=\setupstar,label=(\roman*),parsep=0pt}

\usepackage{fancyhdr}
\usepackage{dopestyle}

\makeatletter
    \renewcommand\@makefntext[1]{\leftskip=0em\hskip-0em\@makefnmark\,#1}
\makeatother

\surroundwithmdframed[skipabove=0.5\baselineskip, skipbelow=0.5\baselineskip, leftmargin=3pt, rightmargin=3pt, innerleftmargin=7pt, innerrightmargin=7pt, roundcorner=10pt, linewidth=2pt, linecolor=red!40, backgroundcolor=red!5]{headsup}

\surroundwithmdframed[topline=false, bottomline=false, innertopmargin=0pt, innerbottommargin=0pt, innerleftmargin=2pt, innerrightmargin=2pt, linewidth=0.2mm]{quote}

\surroundwithmdframed[topline=false, bottomline=false, innertopmargin=2.5pt, innerbottommargin=2.5pt, innerleftmargin=-10pt, leftmargin=-10pt, innerrightmargin=-10pt, rightmargin=7.7pt, linewidth=0.4mm]{quote}

\NewEnviron{method}{
    \parbox{\textwidth}{
        \textbf{\textsc{Method}}
        \begin{mdframed}[innerleftmargin=4pt,innerrightmargin=4pt,skipabove=3pt,skipbelow=0pt]
            \BODY
        \end{mdframed}
    }
}

\theoremstyle{itcaps}
\newtheorem{theorem}{Theorem}
\newtheorem*{theorem*}{Theorem}
\newtheorem{corollary}[theorem]{Corollary}
\newtheorem{conjecture}{Conjecture}
    \crefname{conjecture}{conjecture}{conjectures}
\newtheorem{lemma}[theorem]{Lemma}
\newtheorem{proposition}[theorem]{Proposition}

\theoremstyle{solved}

\theoremstyle{caps}
\newtheorem{definition}[theorem]{Definition}
\newtheorem{question}[conjecture]{Question}

\newtheorem{exampleprimitive}[theorem]{Example}

    \crefname{exercise}{exercise}{exercises}

\newcounter{problemno}
    \crefname{problemno}{problem}{problems}
    \Crefname{problemno}{Problem}{Problems}

\theoremstyle{remark}
\newtheorem*{remark-primal}{Remark}
\newenvironment{remark}{%
    \begin{remark-primal}
}{%
    \hfill$\lozenge$
    \end{remark-primal}
}

\numberwithin{equation}{section}

\newcommand*{\newword}[2][]{\emph{#2}\index{%
    \ifx&#1&%
       #2%
    \else%
       #1%
    \fi}%
} % formats & indexes new words
\newcommand*{\oldword}[2][]{#2\index{%
    \ifx&#1&%
       #2%
    \else%
       #1%
    \fi}%
} % formats & indexes another instance of a previously indexed word

\newcommand*{\E}{\mathbb{E}}

\DeclareMathOperator{\conv}{conv}

\let\phi\varphi

\let\epsilon\varepsilon

\let\oldchi\chi
\renewcommand{\chi}{\raisebox{1pt}{$\oldchi$}}
\newcommand{\defeq}{\vcentcolon=}

\title{\inserttitle}
\date{}

\allowdisplaybreaks